\documentclass[12pt]{amsart}
\usepackage{amssymb,mathrsfs}
\usepackage[matrix,arrow,curve]{xy}
\usepackage{setspace}
\usepackage{color}
\newtheorem{theorem}[equation]{Theorem}
\newtheorem*{theorem*}{Theorem}
\newtheorem{lemma}[equation]{Lemma}
\newtheorem{corollary}[equation]{Corollary}
\newtheorem{conjecture}[equation]{Conjecture}
\newtheorem{question}[equation]{Question}
\newtheorem{proposition}[equation]{Proposition}

\theoremstyle{definition}
\newtheorem{example}[equation]{Example}
\newtheorem{definition}[equation]{Definition}
\newtheorem*{definition*}{Definition}
\newtheorem{resolution}[equation]{Resolution Procedure}

\theoremstyle{remark}
\newtheorem{remark}[equation]{Remark}

\makeatletter\@addtoreset{equation}{section}
\makeatletter\@addtoreset{section}{part}

\makeatother

\def \CC {\mathcal{C}}

\def \P {\mathbb{P}}

\def \PP {\mathbb{P}}

\def \CC {\mathbb{C}}
\def \ZZ {\mathbb{Z}}

\def \Aff {\mathbb{A}}

\def \ge {\geqslant}
\def \le {\leqslant}

\title{On Hodge numbers of complete intersections\\ and Landau--Ginzburg models}

\author{Victor Przyjalkowski, Constantin Shramov}

\thanks{
This work was performed in Steklov Mathematical Institute and supported by the Russian Science Foundation under grant 14-50-00005.
}

\address{
Steklov Mathematical Institute, 8 Gubkina st., Moscow 119991, Russia,\newline
\phantom{La}Laboratory of Algebraic Geometry, GU-HSE, 7 Vavilova st., Moscow 117312,
Russia}\email{victorprz@mi.ras.ru, costya.shramov@gmail.com}

\begin{document}

\begin{abstract}
We prove that the Hodge number
$h^{1,N-1}(X)$ of
an $N$-dimensional ($N\geqslant 3$) Fano complete intersection $X$
is less by one then
the number of irreducible
components of the central fiber of
(any)
Calabi--Yau compactification
of Givental's Landau--Ginzburg model for $X$.
\end{abstract}

\maketitle

\section{Introduction}
\label{section:intro}


Mirror Symmetry conjectures relate symplectic properties of a
variety $X$ to algebro-geometric properties for its mirror symmetry pair
--- a variety $Y$ (or one-parametric family of Calabi--Yau varieties $Y\to \Aff^1$)
and vice-versa, relate algebro-geometric properties of $X$ to
symplectic properties of $Y$.
Homological Mirror Symmetry (see~\cite{Ko94}) treats mirror correspondence in terms of derived categories.
It associates two categories with each variety or family.
Given a symplectic form on $X$, 
one can associate
a so-called \emph{Fukaya category} $Fuk\,(X)$ with $X$ whose objects are Lagrangian submanifolds with respect to the symplectic form.
The relative version of this category, \emph{a Fukaya--Seidel category} $FS(Y)$ can be associated with $Y$.
On the algebraic side of the picture, there are a derived category of coherent sheaves $D^b(X)$ for $X$ and a derived category
of singularities $D^b_{sing}(Y)$ for $Y$, that is,
a direct sum of categories over all fibers whose objects are complexes of coherent sheaves modulo perfect complexes.
Homological Mirror Symmetry conjecture for Fano varieties predicts that for any Fano manifold $X$ there exists a so-called \emph{Landau--Ginzburg model}~\mbox{$Y\to \Aff^1$} such that their categories are cross-equivalent:
$$
Fuk\,(X)\simeq D^b_{sing}(Y),\ \ \  D^b(X)\simeq FS(Y).
$$


In the last two decades Mirror Symmetry, in particular Homological Mirror Symmetry, was deeply developed
and studied on a lot of examples. We mention Gromov--Witten theory --- numerical reflection of Fukaya
category (see, for instance~\cite{KM94} or Manin's book~\cite{Ma99}), and Quantum Lefschetz Theorem (see~\cite{Gi96}).
In~\cite{Gi96} (see also~\cite{HV00}), Givental suggested the version of Landau--Ginzburg models we will use in the paper.
Mirror Symmetry construction for toric varieties and complete intersections therein one can find in~\cite{Ba94}.
This approach can be generalized to complete intersections in varieties admitting good enough toric degeneration,
see, for instance~\cite{Ba97},~\cite{BCFKS97}, and~\cite{BCFKS98}.
Homological Mirror Symmetry for Fano varieties we are interested in was, sometimes partially,
proved for del Pezzo surfaces (\cite{AKO06}) and toric varieties (\cite{Ab09}).
For non-Fano, noncompact, nonsmooth, or noncommutative cases (some versions of) Mirror Symmetry conjectures were verified or
particular constructions were suggested. Let us mention some of results.
A.\,Polishchuk and E.\,Zaslow in~\cite{PZ98} proved Mirror Symmetry for elliptic curves.
K.\,Hori and C.\,Vafa in~\cite{HV00} considered physical aspects of Mirror Symmetry and provide a lot of constructions. K.\,Fukaya discussed it for abelian varieties in~\cite{F02}. Quartic surface case is considered by P.\,Seidel in~\cite{Se03}. In~\cite{Au07}, D.\,Auroux considered connections between Mirror Symmetry and T-duality. D.\,Auroux, L.\,Katzarkov, and D.\,Orlov proved Mirror Symmetry for del Pezzo surfaces, weighted projective planes, and their noncommutative deformations in~\cite{AKO06} and~\cite{AKO08}. Genus two case was proved by P.\,Seidel in,~\cite{Se11}; this result was extended to curves of higher genera by A.\,Efimov in~\cite{Ef09}.
General type case is considered in~\cite{KKOY09} by A.\,Kapustin, L.\,Katzarkov, D.\,Orlov, and M.\,Yotov.
M.\,Abouzaid and I.\,Smith proved Mirror Symmetry for two-dimensional complex tori equipped with the standard symplectic form
in~\cite{AS10}. Mirror correspondence for punctured spheres was proved by M.\,Abouzaid, D.\,Auroux, A.\,Efimov, L.\,Katzarkov, D.\,Orlov in~\cite{AAEKO13}.

As a consequence of the equivalences above one gets a coincidence of
so-called \emph{noncommutative Hodge structures}.
For definitions and constructions of these
structures, see~\cite{KKP08}.
This effect is called \emph{Mirror Symmetry conjecture of variations of Hodge structures}.
It enables
one to translate the mirror correspondence for Fano varieties to a
quantitative level. 
In this version of Mirror Symmetry, a Landau--Ginzburg model
of an $N$-dimensional Fano variety $X$ is by definition an $N$-dimensional variety~$Y$
with a nontrivial map to $\Aff^1$ whose fibers are (compact) Calabi--Yau varieties,
such that the regularized quantum differential operator for $X$ coincides with the Picard--Fuchs operator
for~$Y$ (see Section~\ref{section:toric LG} for details).
This type of Mirror Symmetry is much more easy to establish,
and it is also more convenient to extract
some numerical invariants (say, certain Hodge
numbers) of an initial Fano variety from its dual Landau--Ginzburg model.

Our
motivating problem
is as follows: what information about an
initial Fano variety can we get from its
Landau--Ginzburg model? 
Suppose that $X$ is a smooth Fano threefold of Picard rank one. Then
its Landau--Ginzburg model is known in the sense of Mirror Symmetry of variations of
Hodge structures (see~\cite{Prz13}).
Moreover, according to~\cite{DHKLP}, under some mild natural conditions
all Landau--Ginzburg models for a given threefold
are birational in codimension one. In particular, this means
that if a Landau--Ginzburg model (considered as a fibration),
has reducible fibers, then the numbers of irreducible
components of each fiber do not
depend on a particular choice of a Landau--Ginzburg model.
In~\cite{Prz13}, it is proved that $X$ has a unique reducible fiber and its number of components
equals $\dim(J_X)+1$, where~$J_X$ is the intermediate
Jacobian of~$X$.

What happens in the
higher dimensional case?
In~\cite{GKR12} Gross, Katzarkov, and Ruddat consider $N$-dimensional general type hypersurface $X$
in a smooth toric variety and its \mbox{$N$-dimensional} Landau--Ginzburg model $LG(X)$.
They consider a perverse sheaf~$\mathcal F_{LG(X)}$
of vanishing cycles to the central fiber of $LG(X)$ and
construct a spectral sequence for
cohomologies of $\mathcal F_{LG(X)}$. Via this spectral sequence, they relate
the number of irreducible components of the central fiber of $LG(X)$ with the Hodge number $h^{1,N-1}(X)$.

Mirror duality for Calabi--Yau varieties manifests itself as matching of the
Hodge diamond of a Calabi--Yau variety
and $90^\circ$-rotated Hodge diamond of a mirror dual Calabi--Yau variety (see e.\,g.~\cite{Ba94}).
In~\cite{KKP14}, Katzarkov, Kontsevich, and Pantev give an analog of this duality
for Fano varieties and their dual Landau--Ginzburg models. That is, they propose that
for an $N$-dimensional Fano variety $X$ a Hodge number $h^{p,q}(X)$ is the
dimension of the
$(N-p)$-cohomology group of the sheaf
of so-called \emph{$f_X$-adapted logarithmic \mbox{$q$-forms}}
of a dual to $X$ Landau--Ginzburg model $f_X$.
This sheaf consists of forms with a specific 
condition on orders of poles
on divisors forming a fiber of~$f_X$ over infinity.
The dimensions of these cohomology groups are related to the sizes of
the Jordan blocks of the monodromy around
infinity, which in turn gives a condition
on the monodromy around reducible fibers. 

On a quantitative level for a particular
case of~$h^{1,N-1}(X)$,
this could be interpreted as follows.
For a smooth $N$-dimensional Fano variety
$X$, let $LG(X)$ be its $N$-dimensional Landau--Ginzburg model.
Put
\begin{multline*}
k_{LG(X)}= \sharp\mbox{\big(irreducible components of all reducible
fibers of $LG(X)$\big)} - \\
- \sharp \mbox{\big(reducible fibers\big)}.
\end{multline*}


\begin{conjecture}[see~\cite{GKR12}]
\label{conjecture}
For a smooth Fano variety
$X$ of dimension $N\geqslant 3$, one has $h^{1,N-1}(X)=k_{LG(X)}$.
\end{conjecture}


In~\cite{GKR12}, Conjecture~\ref{conjecture}
is proved for hypersurfaces of general type assuming that
the central fibers of their Landau--Ginzburg models are
semistable, that is they have normal crossing singularities.
Moreover,~\cite{GKR12} also treats the case of a cubic threefold,
that is a simplest case of Fano variety such that
the central fiber of the corresponding Landau--Ginzburg model
is not semistable.

The purpose of this paper is to
prove Conjecture~\ref{conjecture}
for Fano complete intersections
and their (fiberwise compactified)
Givental's Landau--Ginzburg models
(not using the normal crossing condition).

\begin{theorem}[Theorem~\ref{theorem:main}]
\label{theorem:intro}
Let $X$ be a smooth Fano complete intersection of dimension~$N$.
Let $LG(X)$ be a Calabi--Yau compactification
of Givental's Landau--Ginzburg model for~$X$.
Then
$$h^{1,N-1}(X)=k_{LG(X)}$$
if $N>2$, and
$$h^{1,1}(X)=k_{LG(X)}+1$$
if $N=2$.
\end{theorem}

The Hodge number of a smooth complete intersection can be computed combinatorially as a dimension of a
particular component of a particular graded ring, see~\cite{Gr69},~\cite{Di95}, and~\cite{Na97}. On the other
hand, there are Givental's suggestions
for Landau--Ginzburg models (of appropriate dimension) for complete intersections, see~\cite{Gi96}.
They can be birationally rewritten as toric Landau--Ginzburg models. We compute their fiberwise Calabi--Yau
compactifications and compare the number of components of their central fibers of the compactifications with
the Hodge numbers. Note that a general Hironaka-style argument shows that the number of components of the central fiber
does not depend on a choice of a Calabi--Yau compactification.

In practice, we construct Calabi--Yau compactifications of Givental's Landau--Ginzburg models in two steps.
First we find a suitable singular (relative) compactification whose total space is a Calabi--Yau variety. Then
we take a crepant resolution of its singularities. (Recall that a birational morphism $\pi\colon Y\to Z$ of normal $\mathbb{Q}$-Gorenstein varieties is crepant
if~\mbox{$\pi^*(K_Z)\sim K_Y$.})

\begin{remark}
Givental's suggestions for Landau--Ginzburg
models exist and can be rewritten in terms
of Laurent polynomials for smooth
Fano complete intersections of Cartier divisors
in weighted projective spaces (see~\cite[Theorem~2.4]{Prz10}) and in Grassmannians of planes (see~\cite[{Theorem 1.1}]{PSh14b} and~\cite[{Theorem 1}]{PSh14a}). For complete intersections in
weighted projective spaces these weak Landau--Ginzburg models
are toric (see~\cite{ILP13}). According to private communications of A.\,Harder this also holds for complete intersections in
Grassmannians as well,see~\cite{DH15}. For some of them,
say, for threefolds, Theorem~\ref{theorem:intro} holds
(see~\cite[Theorem~23]{Prz13}).
In the index one case the existence of crepant
resolution is proved (see \cite[Remark 2.7]{Prz10}).
It would be interesting to prove an analog of Theorem~\ref{theorem:intro} in a general setup (cf.~\cite[Problem~2.6]{Prz10}).
\end{remark}

\medskip

The paper is organized as follows.
In Section~\ref{section:toric LG}, we recall and discuss
toric Landau--Ginzburg models of complete intersections.
In Section~\ref{section:Hodge numbers},
we express the Hodge number of a complete intersection
in the form suitable for our purposes.
In Section~\ref{section:affine}, we compute
compactifications of toric Landau--Ginzburg models locally,
and find contributions of various strata of singularities to
central fibers of such compactifications.
In Section~\ref{section:projective}, we collect
all such contributions and prove the main theorem of the paper.

\smallskip

Unless explicitly stated otherwise, all varieties are assumed to be
smooth, projective and defined over the field~$\CC$ of complex
numbers.

\medskip

The authors are grateful to V.\,Golyshev, D.\,Orlov and Yu.\,Prokhorov for useful discussions,
A.\,Iliev
for pointing out the paper~\cite{Na97}, L.\,Katzarkov and H.\,Ruddat
for an explanation of a conjecture about relation of the Hodge number and the number of components of fibers in higher dimensional case, and T.\,Pantev for an explanation of Hodge numbers duality for Fano variety and
its Landau--Ginzburg model. Special thanks go to the referees who helped to make the paper
more clear and readable.

\section{Toric Landau--Ginzburg models}
\label{section:toric LG}
Complete intersections are the initial and
one of the most studied examples
of Mirror Symmetry correspondence.
In 1990s Givental computed
their Gromov--Witten invariants using
Quantum Lefschetz Theorem (see~\cite{Gi96}).
The so-called \emph{mirror map} is defined via a generating series of their one-pointed Gromov--Witten invariants ($I$-series). Givental in~\cite{Gi96} defined Landau--Ginzburg models for complete intersections.
It turns out that periods of these Landau--Ginzburg models coincide with regularized Givental's $I$-series.
This phenomenon is called \emph{Mirror Symmetry conjecture of variations of Hodge structures}.
Moreover, a reformulation of Givental's models in terms of Laurent polynomials enables one to
relate them to toric degenerations and gives a powerful tools for effective calculations.

Let us give more details. Consider a smooth Fano complete intersection $X\subset \PP^{N+k}$ of hypersurfaces
of degrees $d_1,\ldots,d_k$. Let $i(X)=N+k+1-\sum d_i$. By~\cite{Gi96}, a \emph{regularized generating series of one-pointed Gromov--Witten invariants
with descendants} is a series
$$
I^X_{0}=e^{-\alpha t}\cdot\left(1+\sum_{d>0} (di(X))!\langle\tau_{di(X)-2} \mathbf 1\rangle_{di(X)}
\cdot t^{di(X)}\right),
$$
where $\alpha=d_1!\cdot\ldots\cdot d_k!$ for $i(X)=1$ and $\alpha=0$ otherwise, and
$\mathbf 1$ is dual to a fundamental class of~$X$.
This series uniquely determines first and second Dubrovin's connections
(or \emph{equations of type~DN}, see~\cite{GS07} and~\cite{Prz07}) and
all Gromov--Witten invariants (coming from~$\PP^{N+k}$) of~$X$, see~\cite{KM94} and~\cite{Prz07}.
According to Givental it equals
$$
I^X_{0}=\sum_{d\geqslant 0} \frac{\left(di(X)\right)!\cdot(dd_1)!\cdot\ldots\cdot(dd_k)!}{(d!)^{N+k+1}}t^{di(X)}.
$$

Consider a Laurent polynomial $f$ in $N$ variables $x_1,\ldots,x_N$. Let $\phi_f(i)$
be the constant term (i.\,e., the coefficient at $x_1^0\cdot \ldots
\cdot x_N^0$) of $f^i$, and define \emph{the constant term series} for $f$ by
$$
\Phi_f=\sum_{i=0}^\infty \phi_f(i)\cdot t^i.
$$
It turns out that, under mild conditions, this series is a period of a family of hypersurfaces in a torus given by
$f$, see, for example,~\cite[Proposition~2.3]{Prz08}.

\begin{definition}
A Laurent polynomial $f$ is called \emph{toric Landau--Ginzburg model}
for $X$ if
\begin{itemize}
  \item (Period condition)
$I^X_{0}=\Phi_f$ up to a linear change of variables.
  \item (Calabi--Yau condition)
There exists a fiberwise compactification of a family
$$f\colon (\CC^*)^N\to \CC$$
whose total space is a (noncompact) smooth Calabi--Yau
variety $LG(X)$. Such compactification is called \emph{a Calabi--Yau compactification}.
  \item (Toric condition) There is a 
degeneration $X\rightsquigarrow T$ to a toric variety~$T$
whose fan polytope (i.\,e., the convex hull of integral generators of rays
of a fan) corresponding to $T$ coincides with the Newton polytope (i.\,e., the convex hull of nonzero coefficients) of $f$.
\end{itemize}
\end{definition}

Givental's Landau--Ginzburg models are certain affine varieties. However, they can be birationally rewritten as Laurent
polynomials
$$
f_{X}=\frac{\prod_{i=1}^k(x_{i,1}+\ldots+x_{i,d_i-1}+1)^{d_i}}{\prod_{i=1}^k \prod_{j=1}^{d_i-1} x_{i,j}\prod_{j=1}^{i(X)-1} y_j}+y_1+\ldots+y_{i(X)-1}
$$
(see~\cite[\S 3.2]{Prz13}).

\begin{theorem}[\cite{Prz13} and~\cite{ILP13}]
The polynomial $f_X$ is a toric Landau--Ginzburg model for $X$.
\end{theorem}

\section{Computing Hodge numbers}
\label{section:Hodge numbers}

Consider a projective space $\PP^{N+k}$ with homogeneous coordinates
$z_1,\ldots,z_{N+k+1}$. Let~\mbox{$f_1,\ldots, f_k$} be homogeneous
polynomials of degrees $d_1,\ldots, d_k$
in $z_1,\ldots,z_{N+k+1}$.
Let
$$
F=F(f_1,\ldots,f_k)=
w_1f_1+\ldots+w_kf_k\in S=\CC[z_1,\ldots,z_{N+k+1}, w_1,\ldots,w_k].
$$
Denote an ideal in $S$ generated by
$$
\frac{\partial F}{\partial w_1},\ldots, \frac{\partial F}{\partial w_k},
\frac{\partial F}{\partial z_1},\ldots,\frac{\partial F}{\partial z_{N+k+1}}
$$
by $J(F)$. Put
$$R=R(f_1,\ldots,f_k)=S/J(F).$$
This ring is bigraded by $\deg (z_s)=(0,1)$ and $\deg(w_j)=(1,-d_j)$.

Although our main character will be (the graded components of)
the ring $R$, for some computations we will need another auxiliary
ring.
Let $J'(F)\subset J(F)$ be an ideal in $S$ generated by
$$\frac{\partial F}
{\partial w_1},\ldots,
\frac{\partial F}{\partial w_{k}}.
$$
Put $$R'=R'(f_1,\ldots,f_k)=S/J'(F),$$ and consider a bigrading on the ring $R'$
given by $\deg (z_s)=(0,1)$ and $\deg(w_j)=(1,-d_j)$.
Then $R$ is a quotient of $R'$ by the ideal $\widehat{J}(F)\subset R'$ generated by the remaining
partial derivatives
$$
\frac{\partial F}{\partial z_1},\ldots,
\frac{\partial F}{\partial z_{N+k+1}},
$$
and the natural homomorphism $R'\to R$ respects the grading.

\begin{lemma}\label{lemma:dim-R-prime-constant}
The graded vector space $R'$ depends only on $d_i$'s
provided that
$f_i$'s define a complete intersection (in a scheme-theoretic sense).
\end{lemma}
\begin{proof}
Induction in $k$.
\end{proof}

\begin{remark}
The assumption of Lemma~\ref{lemma:dim-R-prime-constant}
holds, in particular, for the collection of polynomials
$f_1=z_1^{d_1}, \ldots, f_k=z_k^{d_k}$.
\end{remark}

\begin{remark}
Note that the direct analog of Lemma~\ref{lemma:dim-R-prime-constant}
fails for the ring $R$
(cf. the proof of Lemma~\ref{lemma:R-vs-R-prime}). For example, if $k=1$ and
$$f_1=z_1^{N+1}+\ldots+z_{N+2}^{N+1},$$
then
$$\dim\big(R_{1, -1}(f_1)\big)={2N+1\choose N+1}-N-2,$$
while
$$\dim\big(R_{1, -1}(z_1^{N+1})\big)={2N+1\choose N+1}-1.$$
\end{remark}

From now on, we assume that $X$ is a smooth
Fano complete intersection of the hypersurfaces
defined by the polynomials $f_i$.
Let $h_{pr}^{N-p,p}(X)$ be primitive
middle Hodge numbers of~$X$. In our case, we have $$h_{pr}^{N-p,p}(X)=h^{N-p,p}(X)$$ for $2p\neq N$ and
$$h_{pr}^{N-p,p}(X)=h^{N-p,p}(X)-1$$ otherwise.

Let
$$i(X)=N+k+1-\sum\limits_{t=1}^k d_t\ge 1$$
denote the index of $X$.

The middle Hodge numbers of $X$ can be computed via the dimensions of the graded components of the ring $R$.

\begin{theorem}[{see~\cite{Di95},~\cite{Gr69},~\cite[{Proposition~2.16}]{Na97}}]
\label{theorem:middle-Hodge-numbers}
One has
$$
h_{pr}^{N-p,p}(X)=\dim\big(R_{p,-i(X)}\big).
$$
\end{theorem}
One can also find another formula for Hodge numbers (due to Hirzebruch)
in~\cite[{Theorem 22.1.1}]{Hi66}. 

Theorem~\ref{theorem:middle-Hodge-numbers}
enables us to give an explicit
formula for the middle Hodge number~\mbox{$h_{pr}^{1,N-1}(X)$}.

\begin{lemma}\label{lemma:R-vs-R-prime}
Suppose that $i(X)\ge 2$.
Then
$$\dim(R_{1,-i(X)})=\dim(R'_{1,-i(X)}).$$
Furthermore, if $i(X)=1$, one has
$$\dim(R_{1,-1})=\dim(R'_{1,-1})-(N+k+1).$$
\end{lemma}
\begin{proof}
The ideal $\widehat{J}(F)\subset S'$ is generated by its homogenous component of degree $(1,-1)$. Thus, if $\widehat{J}(F)_{p,q}$ is a nontrivial
homogenous component of $\widehat{J}(F)$ of degree $(p,q)$ then either $p\ge 2$ or $q=-1$. This means that for any $i\ge 2$
one has $\widehat{J}(F)_{1,-i}=0$, so that $R'_{1,-i}$ is isomorphically projected on to $R_{1,-i}$.

To prove the second assertion, note that
the derivatives
$$F_{z_s}=\frac{\partial F}{\partial z_s}\in R'$$
have degrees
$$\deg(F_{z_s})=\deg(F)-\deg(z_s)=
(1,0)-(0,1)=(1,-1).$$
Therefore, the difference
$$\delta(f_1,\ldots,f_k)=\dim\big(R'_{1,-1}(f_1,\ldots,f_n)\big)-
\dim\big(R_{1,-1}(f_1,\ldots,f_n)\big)$$
equals the dimension of the subspace of
$R'_{1,-1}(f_1,\ldots,f_n)$ spanned by the
polynomials~$F_{z_s}$.
By Lemma~\ref{lemma:dim-R-prime-constant},
the dimension $\dim\big(R'_{1,-1}(f_1,\ldots,f_n)\big)$
does not depend on $f_1, \ldots, f_k$ (provided
that the corresponding variety is a complete intersection).
Similarly, the Hodge numbers of a smooth complete intersection
also do not depend on $f_1, \ldots, f_k$.
Thus, Theorem~\ref{theorem:middle-Hodge-numbers} implies that
to compute $\delta(f_1,\ldots,f_k)$ we may choose
the polynomials $f_1,\ldots,f_k$ as we want provided that the complete intersection remains smooth.

Suppose that $d_1$ is minimal
among the degrees $d_j$,
choose
$$f_1=z_1^{d_1}+\ldots+z_{N+k+1}^{d_1},$$
and choose $f_2,\ldots, f_k$ so that
the variety defined by the equations $f_1=\ldots=f_k=0$
is a smooth complete intersection of the hypersurfaces
$f_j=0$ (actually, the latter assumption will not be used at all).
We claim that
the polynomials $$F_{z_s}\in R'_{1,-1}(f_1,\ldots,f_k)$$ are linearly
independent.
Suppose that they are not, that is, for some
$\lambda_1,\ldots,\lambda_{N+k+1}\in\CC$ one has
$$\sum\limits_{s=1}^{N+k+1}\lambda_s F_{z_s}=0\in R'_{1,-1}(f_1,\ldots,f_n).$$
Taking a coefficient at $w_1$, we obtain
$$
\sum\limits_{s=1}^{N+k+1}\lambda_s \frac{\partial f_1}{\partial z_s}
=0\in R'_{1,-1}(f_1,\ldots,f_n).$$
Since
$$\deg\Big(\frac{\partial f_1}{\partial z_s}\Big)=d_1-1<d_1=
\min\{\deg(f_i)\},$$
we conclude that
$$\sum\limits_{s=1}^{N+k+1}\lambda_s \frac{\partial f_1}{\partial z_s}
=0$$
in $\CC[z_1,\ldots,z_{N+k+1}]$.
After a substitution
$$
f_1=z_1^{d_1}+\ldots+z_{N+k+1}^{d_1}
$$
we end up with the equality
$$\sum\limits_{s=1}^{N+k+1}\lambda_s z_s^{d_1-1}=0$$
in $\CC[z_1,\ldots,z_{N+k+1}]$, which gives a contradiction.
\end{proof}

\begin{proposition}\label{proposition:formula-for-R}
If $i(X)\ge 2$, then
\begin{equation*}
\dim\big(R_{1,-i(X)}\big)=
\sum\limits_{j=1}^k\sum\limits_{I\subset I_k}(-1)^{k-|I|}
{\left(\sum_{s\in I} d_s\right)+d_j-1\choose
N+k}.
\end{equation*}

If $i(X)=1$, then
\begin{equation*}
\dim\big(R_{1,-i(X)}\big)=
-(N+k+1)+\sum\limits_{j=1}^k\sum\limits_{I\subset I_k}(-1)^{k-|I|}
{\left(\sum_{s\in I} d_s\right)+d_j-1\choose
N+k}.
\end{equation*}
\end{proposition}

\begin{proof}
Let $\Delta(d, m)$ denote the dimension
of the vector space of homogeneous polynomials of degree $d$ in $m$
variables. Then
$$\Delta(d,m)={d+m-1\choose m-1}.$$

To start with, let us compute the dimension $\Delta_j$ of the vector space
of homogeneous polynomials in the variables $z_1,\ldots, z_{N+k+1}$
of degree $d_j-i(X)$ that are not divisible by any of
the polynomials $f_t$, $1\le t\le k$. Put $I_k=\{1,\ldots, k\}$.
One has
\begin{multline*}
\Delta_j=
\sum\limits_{I\subset I_k}(-1)^{|I|}
\Delta\left(d_j-i(X)-\left(\sum_{s\in I} d_s\right), N+k+1\right)=\\=
\sum\limits_{I\subset I_k}(-1)^{|I|}
{\left(\sum_{s\in I_k\setminus I} d_s\right)+d_j-1\choose
N+k}=\\=
\sum\limits_{I\subset I_k}(-1)^{k-|I|}
{\left(\sum_{s\in I} d_s\right)+d_j-1\choose
N+k}.
\end{multline*}

Recall that
$$\dim\big(R'_{1,-i(X)}(f_1,\ldots,f_k)\big)=
\dim\big(R'_{1,-i(X)}(z_1^{d_1},\ldots, z_k^{d_k})\big)$$
by Lemma~\ref{lemma:dim-R-prime-constant}.
On the other hand, it is straightforward to see
that
$$\dim\big(R'_{1,-i(X)}(z_1^{d_1},\ldots, z_k^{d_k})\big)=
\sum\limits_{j=1}^k \Delta_j.$$
Finally, by Lemma~\ref{lemma:R-vs-R-prime},
one has
$$\dim\big(R_{1,-i(X)}(f_1,\ldots,f_k)\big)=
\dim\big(R'_{1,-i(X)}(f_1,\ldots,f_k)\big),$$
if $i(X)\ge 2$, and
$$\dim\big(R_{1,-i(X)}(f_1,\ldots,f_k)\big)=
\dim\big(R'_{1,-i(X)}(f_1,\ldots,f_k)\big)-(N+k+1),$$
if $i(X)=1$, which completes the proof.
\end{proof}

\begin{corollary}
\label{corollary:Hodge-number-hypersurface}
Let $k=1$.
Then
$$
h_{pr}^{1,N-1}(X)={2d-1\choose N+1}
$$
if $d\leqslant N$ and
$$
h_{pr}^{1,N-1}(X)={2N+1\choose N+1}-N-2
$$
if $d=N+1$.
\end{corollary}
\begin{proof}
Apply Proposition~\ref{proposition:formula-for-R} for $k=1$ and $d_1=d$.
\end{proof}

The formulas from Proposition~\ref{proposition:formula-for-R}
are stated in the
form that enables us to prove our main result.
On the other hand, it may be not really convenient for
computations. For this purpose, one can use the
following formulas (that we will not need in our proof).

\begin{proposition}
One has\footnote{In the published version of this paper there is a misprint in summation indices here.}
\begin{equation*}
\dim\big(R_{1,-i(X)}\big)=
\sum\limits_{j=1}^k\sum\limits_{i_1=0}^{d_1-1}\dots \sum\limits_{i_k=0}^{d_k-1}
{\sum_{t=1}^k (d_t-i_t)+d_j-k-1\choose N}
\end{equation*}
for $i(X)\ge 2$, and
\begin{equation*}
\dim\big(R_{1,-i(X)}\big)=
-(N+k+1)+
\sum\limits_{j=1}^k\sum\limits_{i_1=0}^{d_1-1}\dots \sum\limits_{i_k=0}^{d_k-1}
{\sum_{t=1}^k (d_t-i_t)+d_j-k-1\choose N}
\end{equation*}
for $i(X)=1$.
\end{proposition}
\begin{proof}
By Lemma~\ref{lemma:dim-R-prime-constant}, one has
$$\dim\big(R'_{1, -i(X)}(f_1,\ldots,f_k)\big)=
\dim\big(R'_{1,-i(X)}(z_1^{d_1},\ldots, z_k^{d_k})\big).$$
We start with computing the latter dimension.
The basis of $R'_{1,-i(X)}(z_1^{d_1},\ldots, z_k^{d_k})$
may be chosen to consist of monomials of the form
$$M_{j,i_1,\ldots,i_k}=w_jz_1^{i_1}\dots z_{N+k+1}^{i_{N+k+1}},$$
where $1\le j\le k$, the inequalities $0\le i_t\le d_t-1$ hold
for $1\le t\le k$, and
$$\sum\limits_{r=1}^{N+k+1} i_r=d_j-i(X)=
\big(\sum\limits_{t=1}^k d_t\big)+d_j-N-k-1.$$
Now one has
$$\dim\big(R'_{1,-i(X)}(z_1^{d_1},\ldots,z_k^{d_k})\big)=
\sum\limits_{j=1}^k\sum\limits_{i_1=0}^{d_1-1}\dots \sum\limits_{i_k=0}^{d_k-1}
\Delta\Big(\big(\sum_{t=1}^k d_t-i_t\big)+d_j-N-k-1, N+1\Big).$$

Suppose that $i(X)\ge 2$.
By Lemma~\ref{lemma:R-vs-R-prime},
one has
$$
\dim\big(R_{1,-i(X)}(f_1,\ldots,f_k)\big)=
\dim\big(R'_{1,-i(X)}(f_1,\ldots,f_k)\big),
$$
and the assertion follows by the above computation.

Suppose now that $i(X)=1$.
By Lemma~\ref{lemma:R-vs-R-prime},
one has
$$
\dim\big(R_{1,-1}(f_1,\ldots,f_k)\big)=
\dim\big(R'_{1,-1}(f_1,\ldots, f_k)\big)-(N+k+1),
$$
and the assertion follows in a similar way.
\end{proof}

\section{Local resolutions}
\label{section:affine}

We start with an easy but useful combinatorial observation.

\begin{lemma}
\label{lemma:binomial}
Let $d_1,\ldots,d_k,e,l\in \ZZ_+$.
Then
$$
\sum\limits_{i_1=0}^{d_1}\dots \sum\limits_{i_k=0}^{d_k}
{d_1\choose i_1}\cdot\ldots\cdot{d_k\choose i_k}\cdot
{e\choose i_1+\ldots+i_k+l}
={d_1+\ldots+d_k+e\choose d_1+\ldots+d_k+l}.
$$
\end{lemma}
\begin{proof}
The proof is elementary (and quite standard), but we include it for
the readers convenience.
Take a set $\Gamma$ of
$$|\Gamma|=d_1+\ldots+d_k+e$$
elements, and divide it into subsets
$$\Gamma=\Gamma_1\sqcup\dots\sqcup\Gamma_{k+1}$$
that contain
$d_1,\ldots, d_k$ and $e$ elements, respectively.
Now to choose a subset $\Gamma'\subset\Gamma$ of
$$|\Gamma'|=d_1+\ldots+d_k+l$$
elements is the same as to choose subsets
$\Gamma_i'\subset\Gamma_i$ for $1\le i\le k+1$ such that
$$\Gamma'=\Big(\big(\Gamma_1\setminus\Gamma_1')\sqcup\ldots\sqcup
\big(\Gamma_k\setminus\Gamma_k'\big)\Big)\sqcup\Gamma_{k+1}'.$$
It remains to put $i_j=|\Gamma_j'|$ for $1\le j\le k$,
and to note that
$$|\Gamma_{k+1}'|=i_1+\ldots+i_k+l.$$
\end{proof}

\medskip
In the rest of this section we will
fix some notation, recall the basics on blow ups and describe resolutions
of some special hypersurfaces.
By $\Aff(\xi_1,\ldots,\xi_n)$, we denote the affine space
with homogeneous coordinates $\xi_1,\ldots,\xi_n$, and
by $\PP(\xi_1:\dots:\xi_n)$ we denote the projective space
with homogeneous coordinates $\xi_1,\ldots,\xi_n$.

Consider an affine hypersurface $L=\{f=0\}\subset \Aff(x_1,\ldots,x_n)$.
Suppose that a linear space $\Lambda=\{x_1=\ldots =x_k=0\}$ is contained in $L$.
The blow up of $L$ in $\Lambda$ is given by the same equation $f=0$ in
$
\Aff(x_1,\ldots,x_n)\times
\PP(x'_1:\dots:x'_k)
$
intersected with
$$
\{x_ix'_j=x'_ix_j\}, \ \ \ \ 1\leqslant i,j\leqslant k.
$$
The local charts of the blow up are $x'_i\neq 0$, $i=1,\ldots,k$.
In these local charts, we write~$x_i$'s
instead of~$x'_i$'s for simplicity;
actually the equation of blown up hypersurface
in the~$x_i$th
local chart is obtained
from the initial equation by changing coordinates
$$
x_i\mapsto x_i, x_j\mapsto x_ix_j, \ \ \ \ 1\leqslant j\leqslant k, j\neq i,
$$
and dividing by the maximal possible
power of $x_i$. The exceptional set
is given by $x_i=0$.
So we use the notation $x_i\neq 0$ for this local chart and consider the
equation of the blow up described above.

\begin{remark}
\label{remark:crepant}
In the sequel, we will need the following general fact about discrepancies of blow ups. Let $A$ be a smooth $n$-dimensional variety, let $L$
be a normal irreducible hypersurface in $A$, and let $\Lambda\subset L$ be a smooth irreducible subvariety of dimension $r$.
Let~\mbox{$\pi\colon \widetilde{A}\to A$} be the blow up of $A$ along $\Lambda$ and let $E$ be the exceptional divisor of $\pi$.
Then
$$
K_{\widetilde{A}}\sim \pi^*K_A+(n-r-1)E.
$$
Denote by $m$ the multiplicity of $L$ in a general point of $\Lambda$.
Let $\widetilde{L}$ be the proper transform of $L$, and assume that $\widetilde{L}$ is normal.
Then
$$
K_{\widetilde{L}}\sim \pi^*K_L+(n-r-1-m)\left.E\right|_{\widetilde{L}}.
$$
In particular, if $n-r-1-m=0$, then the morphism $\pi\colon \widetilde{L}\to L$ is crepant
(note that the divisor $\left.E\right|_{\widetilde{L}}$ may be reducible or nonreduced).
\end{remark}

Let $\bar{d}=(d_1,\ldots,d_k)$ for $k\geqslant 0$ and $d_i>0$.
Consider an affine hypersurface
$L_{\bar{d},s}$
given by the equation
$$a_1^{d_1}\cdot\ldots\cdot a_k^{d_k}=\lambda x_1\cdot\ldots\cdot x_s$$
in the affine space $\Aff(a_1,\ldots,a_k,\lambda, x_1,\ldots,x_s)$
as a family of hypersurfaces in the affine space $\Aff(a_1,\ldots,a_k, x_1,\ldots,x_s)$
parameterized by $\lambda\in \CC$.

In most cases, the hypersurface $L_{\bar{d},s}$ is singular
but it always admits a crepant resolution (see Resolution Procedure~\ref{resolution procedure}).
We are going to compute such resolution
and find the number of components in its
central fiber (i.\,e., the fiber over $\lambda=0$).

\begin{remark}\label{remark:pohuj}
Constructing a resolution dominating two given resolutions,
we immediately obtain
that this number does not depend on a particular resolution
(cf.~\cite[Corollary~21]{Prz13}).
\end{remark}

Our strategy is as follows. The exceptional divisors in the central fiber of our resolution appear
from resolutions with centers in (some of the)
subvarieties
of~$L_{\bar{d},s}$ given by vanishing of one coordinate $a_j$ and several coordinates
$x_i$, which we call special strata. 
We resolve
the initial hypersurface by blowing up
these strata in an accurately chosen order, and count the appearing components of central fibers.
By a happy coincidence, after a blow up the singularities
that we still have to resolve are given in the corresponding
affine charts by the equations of the same type as before,
so that we can apply the same procedure again,
and proceed in this manner to obtain a resolution
(see Resolution Procedure~\ref{resolution procedure}).
Note that in some cases
special strata are not contained in a singular locus of~$L_{\bar{d},s}$.

It will happen that on each step our variety has hypersurface singularities,
and their multiplicity at the stratum we blow up on this step
equals the codimension of the stratum. This means that the blow up is crepant by Remark~\ref{remark:crepant},
and the composition of such blowups is a crepant resolution of singularities.

We start with a construction of a crepant resolution for $L_{\bar{d},s}$.

\begin{resolution}
\label{resolution procedure}
Let us call a pair $w(\bar{d},s)=(s, \sum d_i)$ a \emph{weight} of $L_{\bar{d},s}$.
Let~$X$ be a variety covered by affine charts $U_p$ of type
$L_{\bar{d}_p,s_p}$. We say that they \emph{agree with each other} if for any charts $U_{p}$ and $U_q$
the following property holds: for any variable~$a_i$ in~$U_p$ the closure of the divisor $a_i=0$
in $X$ intersects $U_q$ either by an empty set or
by a divisor $a_{i'}=0$, and
for any variable $x_j$ in $U_p$ the closure of the divisor $x_j=0$
in~$X$ intersects $U_q$
either by an empty set or by a divisor $x_{j'}=0$.
Let us note that for $k=0$ or $s=0$ the variety $L_{\bar{d},s}$ is smooth.
Thus we suppose that $k\geqslant 1$ and $s\geqslant 1$.

Suppose that $d_1\geqslant s$.
Blow up the stratum
$$
\Lambda=\{a_1=x_1=\ldots=x_s=0\}.
$$
In a local chart $a_1\neq 0$, we get a hypersurface given by
$$a_1^{d_1-s}\cdot\ldots\cdot a_k^{d_k}=\lambda x_1\cdot\ldots\cdot x_s.$$
This is a hypersurface of type $L_{(d_1-s,d_2,\ldots,d_k),s}$ of weight
$(s,\sum d_i-s)$ which is lexicographically smaller than $(s, \sum d_i)$.

In the local chart $x_i\neq 0$, we get a hypersurface given by
$$
a_1^{d_1}\cdot\ldots\cdot a_k^{d_k}\cdot a_{k+1}^{d_1-s}=\lambda
x_1\cdot\ldots\cdot x_{i-1}\cdot x_{i+1}\cdot\ldots\cdot x_s,
$$
where we denote $x_i$ by $a_{k+1}$.
This is a hypersurface of type $L_{(d_1,\ldots,d_k,d_1-s),s-1}$ of weight
$(s-1,\sum d_i+d_1-s)$ which is lexicographically smaller than $(s, \sum d_i)$.

Suppose that $d_1 < s$.
Blow up the stratum
$$
\Lambda=\{a_1=x_1=\ldots=x_{d_1}=0\}.
$$
Note that this morphism can be small. This happens, for example, when~\mbox{$d_1=\ldots=d_k=1$}.

In a local chart $a_1\neq 0$, we get a hypersurface given by
$$a_2^{d_2}\cdot\ldots\cdot a_k^{d_k}=\lambda x_1\cdot\ldots\cdot x_s.$$
This is a hypersurface of type $L_{(d_2,\ldots,d_k),s}$ of weight
$(s,\sum d_i-d_1)$ which is lexicographically smaller than $(s, \sum d_i)$.

In the local chart $x_i\neq 0$, we get a hypersurface given by
$$
a_1^{d_1}\cdot\ldots\cdot a_k^{d_k}= \lambda
x_1\cdot\ldots\cdot x_{i-1}\cdot x_{i+1}\cdot\ldots\cdot x_s.
$$
This is a hypersurface of type $L_{\bar{d},s-1}$ of weight
$(s-1,\sum d_i)$ which is lexicographically smaller than $(s, \sum d_i)$.

We claim that the above blow ups are crepant morphisms. Indeed, in each case the multiplicity of singularities of $L_{\bar{d},s}$ in $\Lambda$
equals the codimension of $\Lambda$, so that we can apply Remark~\ref{remark:crepant}.

Note that the local charts we obtain after the blow up agree with each other.
Moreover, the divisor given by the equation $a_i=0$ in one of the charts
corresponds to a divisor given by the same equation $a_i=0$ in any other chart
provided that its intersection with the latter chart is nonempty.

Now apply the above procedure simultaneously in all affine charts
where our equation depends on the variable $a_1$.
Since the charts agree with each other, these blow ups glue together and define
a blow up of the whole (nonaffine) variety.
If the variable $a_1$ does not appear in any of the affine charts,
we shift the numeration of the variables $a_i$ simultaneously in all charts
(using once again that they agree with each other) and reset the procedure with the new $a_1$.
At each step, the weight lexicographically decreases, and thus we arrive to a resolution
of singularities after a finite number of steps.
\end{resolution}

Now we will count the number of exceptional divisors over $\lambda=0$ appearing in the Resolution Procedure~\ref{resolution procedure}
applied to the hypersurface $L_{\bar{d},s}$. We start with a particular case that will finally imply the general one.

Consider an affine hypersurface $L_{d,s}$ for $d\geqslant 1$ and $s\ge 1$
given by the equation
$$a^d=\lambda x_1\cdot\ldots\cdot x_s$$
in the affine space $\Aff(a,\lambda, x_1,\ldots,x_s)$.
We consider it as a family of hypersurfaces
in~$\Aff(a,x_1,\ldots,x_s)$
parameterized by $\lambda\in \CC$.

Let $F(d, s)$ be the number of components (i.\,e., the number of
exceptional divisors plus one)
over $\lambda=0$ in (any) crepant resolution of singularities
of $L_{d,s}$.
In particular, $L_{d,0}$ defines a smooth hypersurface so that $F(d,0)=1$.

The special
strata for $L_{d,s}$
we use in Resolution Procedure~\ref{resolution procedure} are given by an equation
$$a=x_1=\ldots=x_s=0$$
if $d\geqslant s$, and by equations
$$a=x_{i_1}=\ldots=x_{i_d}=0$$
for any subset of indices
$\{i_1,\ldots,i_d\}\subset \{1,\ldots,s\}$ if $d<s$.
Exceptional divisors in the
central fiber with centers on $L_{d,s}$
containing (one of) these deepest strata are of two types:
ones whose centers are intersections of the central fiber with
special strata of lower codimension,
and ones whose centers coincide with the intersection of
the central fiber and the deepest special strata.
We denote the number of exceptional divisors of the second kind by~$G(d,s)$.
Put in addition $G(d,0)=1$ and put $F(r,s)=0$ for $r\leqslant 0$.

\begin{lemma}
\label{lemma:G-in-terms-of-F}
One has $G(d,s)=F(d-s,s)$.
\end{lemma}

\begin{proof}
We follow Resolution Procedure~\ref{resolution procedure} to count exceptional divisors in the central fiber.
In particular, we ignore the charts where such divisors do not appear.

Suppose that $d>s$. We say that the stratum given by the equations
$$a=x_1=\ldots=x_s=0$$
is \emph{a canonical stratum}.
Blow it up.
In the local chart $a\neq 0$, we get a hypersurface given by
$$
a^{d-s}=\lambda x_1\cdot\ldots\cdot x_s
$$
together with one exceptional divisor in the central fiber.
This equation defines a variety~$L_{d-s,s}$.
So the total number of exceptional divisors in the central fiber in this local chart equals the number of
exceptional divisors for $L_{d-s,s}$ plus one, that is
$F(d-s,s)$.
In the local chart $x_i\neq 0$, we get a hypersurface given by
$$
x_i^{d-s}a^d=
\lambda x_1\cdot\ldots\cdot x_{i-1}\cdot x_{i+1}\cdot\ldots\cdot x_s.
$$
Exceptional divisors of the resolution of singularities of $L_{d,s}$
given by Resolution Procedure~\ref{resolution procedure}
in this local chart
that lie in the central fiber and
do not intersect the local chart $a\neq 0$ actually lie over the
intersection of the central fiber with the stratum
$$
a=x_1=\ldots= x_{i-1}= x_{i+1}= \ldots=x_s=0,
$$
that is, they come from the stratum of $L_{d,s}$ of lower codimension.
Thus we do not have any additional
contribution to $G(d,s)$ from these divisors.
Therefore, we obtain an equality
$$G(d,s)=F(d-s,s)$$ for $d>s$.

Suppose that $d\leqslant s$.
Then there are several deepest special strata except for the case $d=s$.
Blow up one of these strata, say, given by
$$a=x_1=\ldots=x_d=0.$$
In the local chart $a\neq 0$, we get a hypersurface given by
$$
1=\lambda x_1\cdot\ldots\cdot x_s.
$$
In the neighborhood of exceptional divisor, this hypersurface is smooth.
In the local chart~$x_i\neq 0$, we get a hypersurface given by
$$
a^d=\lambda x_1\cdot\ldots\cdot x_{i-1}\cdot x_{i+1}\cdot\ldots\cdot x_s.
$$
The deepest special strata in this chart are
given by equations
$$
a=x_{i_1}=\ldots=x_{i_d},\quad \{i_1,\ldots,i_d\}\subset \{1,\ldots,s\}\setminus \{i\}.
$$
It means that the blow up of our deepest special
stratum
is an isomorphism in the neighborhood of a general point
of the intersection of the central fiber with such strata,
so that blowing up the deepest special stratum
does not contribute to the set of exceptional divisors
over~\mbox{$\lambda=0$}. Blowing up these strata
one-by-one we see that $G(d,s)=0$ for~$d\leqslant s$.
\end{proof}


\begin{lemma}
\label{lemma:F-in-terms-of-G}
For all $d,s\ge 0$, one has
\begin{equation*}
F(d,s)=\sum_{i=0}^s {s\choose i}G(d,i).
\end{equation*}
\end{lemma}

\begin{proof}
Apply Resolution Procedure~\ref{resolution procedure}
and sum up the numbers of exceptional divisors coming from
each stratum.
\end{proof}

\begin{proposition}\label{proposition:G}
For all $d\ge 1$ and $s\ge 0$, one has
\begin{equation*}
G(d,s)={d-1\choose s}.
\end{equation*}
\end{proposition}

\begin{proof}
We have $G(d,0)=1$ by definition for all $d\ge 1$.
Also, if $d\le s$, the assertion holds by Lemma~\ref{lemma:G-in-terms-of-F}.
The remaining part of the assertion
is proved by induction.
Assume that~\mbox{$d>s>0$}, and that the formula holds
for all $G(d',s')$ with $d'<d$ and
$s'\le s$.
We compute
\begin{eqnarray*}
G(d,s)=F(d-s,s)=
\sum\limits_{i=0}^s {s\choose i}G(d-s,i)=
\sum\limits_{i=0}^s {s\choose i}{d-s-1\choose i}=
{d-1\choose s},
\end{eqnarray*}
where the first equality
comes from Lemma~\ref{lemma:G-in-terms-of-F},
the second one is an application of Lemma~\ref{lemma:F-in-terms-of-G},
the third one is the induction hypotheses, and the fourth one
is implied by Lemma~\ref{lemma:binomial} with~\mbox{$k=1$, $d_1=d$,}
and $e=d-s-1$.
\end{proof}

\begin{example}
\label{example:du Val}
Consider a hypersurface $L_{d+1,1}$. It is given by an equation $a^{d+1}=\lambda x$ so it defines (a cone over) a usual du Val singularity of type $A_{d}$. As a first step of Resolution Procedure~\ref{resolution procedure}, we blow up a stratum $a=x=0$ obtaining one exceptional
divisor in the central fiber.
In a local chart $x\neq 0$, the result of the blow up is smooth.
In another chart, we get a hypersurface $\{a^{d}=\lambda x\}$, that is a du Val singularity of type $A_{d-1}$. Thus following Resolution Procedure~\ref{resolution procedure} we resolve the initial singularity in $d$ steps obtaining a single exceptional divisor
in each step. Therefore, we have $F(d+1,1)=d+1$.

Note that if $d\geqslant 2$, then we can get the same resolution in more fast and standard way blowing up the singular locus $a=\lambda=x=0$.
This gives us two exceptional divisors at once and leaves us with a du Val singularity of type $A_{d-2}$.
\end{example}

In practice, we often avoid applying Resolution Procedure~\ref{resolution procedure} literally
and make shortcuts like one described in Example~\ref{example:du Val} to obtain a resolution in fewer steps.

\begin{example} Consider the hypersurfaces $L_{3,s}$.
\label{example:local cubics}
\begin{enumerate}
  \item Let $s=1$.
  A hypersurface $L_{3,1}$ is given by equation $a^3=\lambda x$. This is a du Val singularity of type $A_2$.
So $F(3,1)=3$ by Example~\ref{example:du Val}.

  \item Let $s=2$. A hypersurface $L_{3,2}$ is given by equation $a^3=\lambda xy$.
  Its singularity locus consists of
  three lines in $\Aff (a,\lambda,x,y)$, two of which lie in the central fiber, and these three lines
  intersect in one point. Blow up the ``horizontal'' line $a=x=y=0$.
  In a local chart $a\neq 0$, one gets a smooth hypersurface
  $a=\lambda xy$ and three exceptional divisors, $E_1=\{a=\lambda=0\}$, $H_1=\{a=x=0\}$, and $H_2=\{a=y=0\}$.
  The divisor $E_1$ lies in the central fiber.
In the local chart, $x\neq 0$ one gets a hypersurface
$$a^3x=\lambda y.$$
The exceptional divisors in this chart are $E_1=\{x=\lambda\}$ and $H_2=\{x=y=0\}$.
  In the same way, one has exceptional divisors $E_1$ and $H_1$ in a local chart $y\neq 0$.

  Take a hypersurface $L_{(3,1),1}$ given by equation $a^3x=\lambda y$.
  Instead of following Resolution Procedure~\ref{resolution procedure}, we apply a trick similar to one described
  in Example~\ref{example:du Val}. The singularity locus of our hypersurface is the line
  $$a=\lambda=y=0.$$ Blow it up. In a local chart $a\neq 0$, one gets a hypersurface
  $ax=\lambda y$ and two exceptional divisors $E_2=\{a=\lambda=0\}$ and $E_3=\{a=y=0\}$, both of which lie in the central fiber.
  The singularities in this local chart are just one ordinary double point admitting a small resolution.
  In two other local charts, $\lambda\neq 0$ and $y\neq 0$ one gets smooth hypersurface containing exceptional
  divisors $E_3$ in the first chart and~$E_2$ in the second one.

  In the same way, one gets two exceptional divisors $E_4$ and $E_5$ from the hypersurface $a^3y=\lambda x$.
  Thus, divisors $E_1,\ldots,E_5$
  together with the initial central fiber form the central fiber of the resolution and
  $F(3,2)=6$.

  \item Let $s=3$. A hypersurface $L_{3,3}$ is given by the equation $a^3=\lambda xyz$.
  Blow up the canonical stratum
  $$a=x=y=z=0.$$
  In a local chart $a\neq 0$, one gets a smooth hypersurface
  $1=\lambda xyz$ and one ``horizontal'' exceptional divisor. In a local chart $x\neq 0$, one gets a hypersurface
  $a^3=\lambda y z$ containing the same exceptional divisor. Resolving this hypersurface, one gets
  five exceptional divisors in the central fiber, four of which are common with other local charts $y\neq 0$ and $z\neq 0$.
  In total, one gets 10 exceptional divisors and $F(3,3)=10$.

  \item Let $s>3$.
  A hypersurface $L_{3,s}$ is given by the equation $$a^3=\lambda x_1\cdot\ldots\cdot x_s.$$
  Blow up the stratum
  $$a=x_i=x_j=x_k=0.$$
  In a local chart $a\neq 0$, one gets a smooth hypersurface
  and no exceptional divisors in the central fiber. In a local chart $x_i\neq 0$,
  one gets an equation
  $$a^3=\lambda x_1\cdot\ldots\cdot x_{i-1}\cdot x_{i+1}\cdot\ldots\cdot x_s$$ and no central
  exceptional divisors as well. The same happens in other local charts $x_j\neq 0$ and $x_k\neq 0$.
  One can carry on with this procedure decreasing $s$ down to $2$. Finally, one gets
  $\frac{s(s-1)}{2}$ exceptional divisors in the central fiber coming from $\frac{s(s-1)}{2}$ strata
  of type $a=x_i=x_j=0$, together with $2s$ exceptional divisors coming from strata $a=x_i=\lambda=0$ and the initial central fiber.
  In total, one gets
  $$F(3,s)=\frac{s(s-1)}{2}+2s+1={s\choose 2}.$$
\end{enumerate}
\end{example}

Let $F(\bar{d},s)$ be the number
of components in the central fiber
of a crepant resolution of~$L_{\bar{d},s}$
(as above, $F(\bar{d},s)$ does not depend on the choice of a
crepant resolution).

\begin{lemma}
\label{lemma:multi-F}
For any $\bar{d}$, one has
$$
F(\bar{d},s)=\sum_{i=1}^k {d_i+s-1 \choose s}.
$$
\end{lemma}

\begin{proof}
One can resolve singularities in the same way as in
Resolution Procedure~\ref{resolution procedure}. That is, first resolve
singularities lying over $a_1=0$,
then resolve ones lying over $a_2=0$, etc. Finally, we get
$$F(\bar{d},s)=\sum_{i=1}^k F(d_i,s),$$
and the assertion follows by Lemma~\ref{lemma:G-in-terms-of-F}
and Proposition~\ref{proposition:G}.
\end{proof}

\section{Global resolutions}
\label{section:projective}

Consider a complete intersection $X\subset \PP^{N+k}$ of hypersurfaces of degrees $d_1,\ldots,d_k$.
Recall that its Givental's toric Landau--Ginzburg model
is
$$
f_{X}=\frac{\prod_{i=1}^k(x_{i,1}+\ldots+x_{i,d_i-1}+1)^{d_i}}{\prod_{i=1}^k \prod_{j=1}^{d_i-1} x_{i,j}\prod_{j=1}^l y_j}+y_1+\ldots+y_l
$$
with
\begin{equation*}
\label{eq:l}
l=N+k-\sum d_i=i(X)-1\ge 0
\end{equation*}
(see~\cite[\S 3.2]{Prz13} and~\cite{ILP13}).

Consider a fiberwise compactification of the hypersurface
given by $f_X=\lambda$ that is a  (singular) hypersurface
$LG_s(X)$ in
$\PP^{d_1-1}\times\dots\times \PP^{d_k-1}\times \PP^l\times\Aff^1$
given by an equation
$$y_0^{l+1}\prod_{i=1}^k(x_{i,1}+\ldots+x_{i,d_i})^{d_i}=(\lambda y_0-y_1-\ldots-y_l){\prod_{i=1}^k \prod_{j=1}^{d_i} x_{i,j}\prod_{j=1}^l y_j},$$
where $x_{i,1},\ldots,x_{i,d_i}$ are homogeneous coordinates in $\PP^{d_i-1}$, and $y_0,\ldots, y_l$ are homogeneous coordinates in $\PP^{l}$.
We obtain (an open) Calabi--Yau variety.
Applying Resolution Procedure~\ref{resolution procedure}
we are going to construct a Calabi--Yau compactification $LG(X)$
of the toric Landau--Ginzburg model
(cf.~\cite[Proposition~11]{Prz13}).
The only fiber of the compactified fibration
that may be reducible is the fiber over $\lambda=0$.
By Remark~\ref{remark:pohuj}, the number of its irreducible components
(which equals $k_{LG(X)}+1$ by definition of $k_{LG(X)}$,
see Section~\ref{section:intro})
does not depend on the choice of a Calabi--Yau compactification.

\begin{theorem}
\label{theorem:main}
There exists a Calabi--Yau compactification $LG(X)$ of $f_{X}$
such that
$$h^{1,N-1}(X)=k_{LG(X)}$$
if $N>2$, and
$$h^{1,1}(X)=k_{LG(X)}+1$$
if $N=2$.
\end{theorem}

\begin{proof}
In the neighborhood of the locus $y_0=0$, the hypersurface $LG_s(X)$ is analytically equivalent to a hypersurface
covered by local charts of type
$$
u_0^{l+1}\cdot u_1^{d_1}\cdot\ldots \cdot u_k^{d_k}= v_1\cdot\ldots\cdot v_s
$$
in some variables $u_0,\ldots,u_k, v_1,\ldots,v_s$, where $s<N$.
Each of these charts has a crepant resolution (see Resolution Procedure~\ref{resolution procedure}).
Moreover, these resolutions agree to each other. Finally, the equations of the hypersurface in these charts
do not depend on $\lambda$ so there are no exceptional divisors in the central fiber.

Thus one can assume $y_0\neq 0$. Similarly, in a local chart $y_0\neq 0$, $y_i\neq 0$, $i\in \{1,\ldots,l\}$,
there is a crepant resolution without any exceptional divisors in the central fiber.
This means that one can consider only a neighborhood of the locus
$$y_0\neq 0, y_1=\ldots=y_l=0.$$
In particular, in a neighborhood of a point of a central fiber
dominated by an exceptional divisor the equation of $LG_s(X)$
can be locally analytically rewritten as an equation
\begin{equation*}\label{eq:LG-hypersurface-simplified}
a_1^{d_1}\cdot\ldots\cdot a_k^{d_k}=\lambda y_1\cdot\ldots\cdot
y_l\cdot \prod\limits_{i=1}^k \prod_{j=1}^{d_i}x_{i,j}
\end{equation*}
with
$$a_i=x_{i,1}+\ldots+x_{i,d_i},$$
that defines a hypersurface in
$\PP^{d_1-1}\times\dots\times \PP^{d_k-1}\times \Aff^l\times\Aff^1$.
Furthermore, exceptional divisors over $\lambda=0$
lie over general points of strata that are given by vanishing
$$a_r=y_1=\ldots=y_l=0$$
for some $1\leqslant r\leqslant k$
together with vanishing of some
variables $x_{i,j_1},\ldots,x_{i,j_{s_i}}$
for various~\mbox{$1\leqslant i\leqslant k$},
such that for $i\neq r$ one has $s_i\leqslant d_{i}-1$, and for $i=r$
one has $s_i\leqslant d_{i}-2$, and the total number $\sum s_i$
of the vanishing variables $x_{i,j}$ is at least one if $l=0$.
For any $i\in \{1,\ldots,k\}$, a projective space $\PP^{d_i-1}$ is covered by local charts $x_{i,r}\neq 0$. In each
of these charts take $a_i,x_{i,1},\ldots,x_{i,r-1}, x_{i,r+1},\ldots,x_{i,d_i}$ as coordinates.
Analytically in these local charts, taken over all ${i}$, the hypersurface $LG_s(X)$ is of type
$L_{\bar{d},s}$ for $\bar{d}=(l+1,d_1,\ldots,d_k)$ and $s\leqslant l+\sum d_i-k=N$.
To compute the total number of exceptional divisors
over a central fiber of~$LG_s(X)$, we need to sum up the contributions
of various canonical strata, see the proof of Lemma~\ref{lemma:G-in-terms-of-F}.

Let  $I_k=\{1,\ldots,k\}$, and for
$1\le j\le k$ let $I_k^j=I_k\setminus\{j\}$.
If $l\ge 1$, then
the canonical strata
are labeled by the choices
of $j\in I_k$, and $i_t$ variables of
$x_{t,1},\ldots,x_{t,d_t}$ for all~$1\le t\le k$,
where $0\le i_t\le d_t-1$ for $t\neq j$ and $0\le i_j\le d_j-2$.
Adding up the contributions of the canonical
strata, we see that
the number of exceptional components over
the central fiber of $LG_s(X)$ equals

\begin{multline*}
\sum\limits_{j=1}^k
\sum\limits_{i_j=0}^{d_1-2}
\sum\limits_{s\in I_k^j}
\sum\limits_{i_s=0}^{d_1-1}
{d_1\choose i_1}\cdot\ldots\cdot{d_k\choose i_k}\cdot
G\big(d_j,\big(\sum_{s\in I_k} i_s\big)+l\big)=\\=
\sum\limits_{j=1}^k
\sum\limits_{i_j=0}^{d_1-2}
\sum\limits_{s\in I_k^j}
\sum\limits_{i_s=0}^{d_1-1}
{d_1\choose i_1}\cdot\ldots\cdot{d_k\choose i_k}\cdot
{d_j-1\choose \left(\sum_{s\in I_k} i_s\right)+l}=\\=
\sum\limits_{j=1}^k
\sum\limits_{i_1=0}^{d_1-1}
\dots
\sum\limits_{i_k=0}^{d_k-1}
{d_1\choose i_1}\cdot\ldots\cdot{d_k\choose i_k}\cdot
{d_j-1\choose \left(\sum_{s\in I_k} i_s\right)+l}=\\=
\sum\limits_{j=1}^k
\sum\limits_{I\subset I_k}
(-1)^{k-|I|}
\left(\sum\limits_{s\in I}
\sum\limits_{i_s=0}^{d_s}
\left(\prod\limits_{s\in I}
{d_s\choose i_s}\right)
{d_j-1\choose
\left(\sum_{s\in I} i_s\right)+
\left(\sum_{s\in I_k\setminus I} d_s\right)+l}
\right)=\\=
\sum\limits_{j=1}^k
\sum\limits_{I\subset I_k}
(-1)^{k-|I|}
{\left(\sum_{s\in I} d_s\right)+d_j-1\choose
\left(\sum_{s\in I_k} d_s\right)+l}=\\=
\sum\limits_{j=1}^k
\sum\limits_{I\subset I_k}
(-1)^{k-|I|}
{\left(\sum_{s\in I} d_s\right)+d_j-1\choose
N+k}.
\end{multline*}
The first equality above
follows from Proposition~\ref{proposition:G}.
The second one follows from the fact that if $l\ge 1$, and $i_s$ is chosen to be
$i_s=d_s-1$, then the corresponding summand equals~$0$ regardless
of the choice of $i_t$ for $t\neq s$.
The third equality is a usual inclusion-exclusion formula.
The fourth equality is an application of Lemma~\ref{lemma:binomial}.
Finally, the fifth equality follows from the definition of $l$.
Components of the central fiber of $LG(X)$ are exceptional divisors and the strict transform of the (irreducible) central fiber of $LG_s(X)$. Thus, the number computed above is exactly $k_{LG(X)}$. Comparing it with the number computed
in Proposition~\ref{proposition:formula-for-R}, one obtains that $k_{LG(X)}=h_{pr}^{1,N-1}$.

If $l=0$, then
the canonical strata
are labeled by the choices
of $j\in I_k$, and $i_t$ variables of
$x_{t,1},\ldots,x_{t,d_t}$ for all $1\le t\le k$,
where $0\le i_t\le d_t-1$ for $t\neq j$, $0\le i_j\le d_j-2$,
and at least one of the numbers $i_1,\ldots,i_k$ is nonzero.
Adding up the contributions of the canonical
strata, we see that
the number of exceptional components over
the central fiber of $LG_s(X)$ equals

\begin{multline*}\label{eq:computation-for-l-0}
-\sum\limits_{j=1}^k {d_1\choose 0}\cdot\ldots\cdot{d_k\choose 0}\cdot
{d_j-1\choose 0}+
\sum\limits_{j=1}^k
\sum\limits_{i_j=0}^{d_1-2}
\sum\limits_{s\in I_k^j}
\sum\limits_{i_s=0}^{d_1-1}
{d_1\choose i_1}\cdot\ldots\cdot{d_k\choose i_k}\cdot
{d_j-1\choose
\sum_{s\in I_k} i_s}=\\=
-k+\sum\limits_{j=1}^k
\sum\limits_{i_j=0}^{d_1-2}
\sum\limits_{s\in I_k^j}
\sum\limits_{i_s=0}^{d_1-1}
{d_1\choose i_1}\cdot\ldots\cdot{d_k\choose i_k}\cdot
{d_j-1\choose \sum_{s\in I_k} i_s}=\\=
-\big(k+\sum\limits_{s\in I_k} d_s\big)+
\sum\limits_{j=1}^k
\sum\limits_{i_1=0}^{d_1-1}
\ldots
\sum\limits_{i_k=0}^{d_k-1}
{d_1\choose i_1}\cdot\ldots\cdot{d_k\choose i_k}\cdot
{d_j-1\choose \sum_{s\in I_k} i_s}=\\=
-\big(k+\sum\limits_{s\in I_k} d_s\big)+
\sum\limits_{j=1}^k
\sum\limits_{I\subset I_k}
(-1)^{k-|I|}
{\left(\sum_{s\in I} d_s\right)+d_j-1\choose
\left(\sum_{s\in I_k} d_s\right)}=\\=
-(N+2k)+
\sum\limits_{j=1}^k
\sum\limits_{I\subset I_k}
(-1)^{k-|I|}
{\left(\sum_{s\in I} d_s\right)+d_j-1\choose
N+k}
\end{multline*}
The first two equalities above
are straightforward.
The third one follows from inclusion-exclusion formula and
Lemma~\ref{lemma:binomial} just as in the case $l\ge 1$.
The fourth equality follows from the definition of $l$. 
Components of the central fiber of $LG(X)$ are exceptional divisors and the strict transforms of the components of the (reducible) central fiber of $LG_s(X)$. The central fiber of $LG_s(X)$ consists of $k$ components. Thus the number computed above is $k_{LG(X)}-k+1$. Comparing it with the number computed
in Proposition~\ref{proposition:formula-for-R} one obtains that $k_{LG(X)}=h_{pr}^{1,N-1}$.

Finally, the statement of the theorem follows from
the fact that
$h_{pr}^{N-p,p}(X)=h^{N-p,p}(X),$
with the only exception
$h_{pr}^{k,k}(X)=h^{k,k}(X)-1$
for $N=2k$.
\end{proof}

\begin{example} \label{example:global cubics}
Consider a cubic hypersurface $X$ in $\P^{N+1}$.
\begin{enumerate}
\item Let $N$ be equal to $2$.
Then
$$
f_X=\frac{(x_1+x_2+1)^3}{x_1x_2}
$$
and $LG_s(X)$ is given by
$$
(x_1+x_2+x_3)^3=\lambda x_1x_2x_3.
$$
The strata we blow up are given by equations
$$x_1+x_2+x_3=x_j=0.$$
In a neighborhood of such stratum $LG_s(X)$ is isomorphic to a hypersurface~\mbox{$a^3=\lambda x_j$.} Thus, $$k_{LG(X)}=3\cdot 2=6$$
by Example~\ref{example:local cubics}(1), and
$$h^{1,1}(X)=7=k_{LG(X)}+1.$$

\item Let $N$ be equal to $3$.
Then
$$
f_X=\frac{(x_1+x_2+1)^3}{x_1x_2y_1}+y_1
$$
and $LG_s(X)$ is given by
$$
y_0^2(x_1+x_2+x_3)^3=(\lambda y_0+y_1)y_1 x_1x_2x_3.
$$
The strata we blow up are given by
$$y_1=x_1+x_2+x_3=x_j=0, \ y_0=1.$$
In a neighborhood of such stratum $LG_s(X)$ is analytically equivalent to a hypersurface
$a^3=\lambda y_1 x_j$.
So there are $F(3,1)=3$ central components over a common strata (given by $a=y_1=0$) of singularities for all of these hypersurfaces
by Example~\ref{example:local cubics}(1), and $G(3,2)=1$ central components (that do not come from higher dimensional strata)
for each of these strata. Thus,
$$k_{LG(X)}=2+3\cdot 1=5$$
and
$$h^{1,2}(X)=5=k_{LG(X)}.$$

\item Let $N$ be equal to $4$.
Then
$$
f_X=\frac{(x_1+x_2+1)^3}{x_1x_2y_1y_2}+y_1+y_2
$$
and $LG_s(X)$ is given by
$$
y_0^3(x_1+x_2+x_3)^3=(\lambda y_0+y_1+y_2)y_1 y_2 x_1x_2x_3.
$$
The strata we blow up are given by
$$y_1=y_2=x_1+x_2+x_3=x_j=0, \ y_0=1.$$
In a neighborhood of such stratum $LG_s(X)$ is analytically equivalent to a hypersurface
$a^3=\lambda y_1 y_2 x_j$.
So there are $F(3,2)=1$ central components over a common strata (given by $a=y_1=y_2=0$) of singularities for all of these hypersurfaces
by Example~\ref{example:local cubics}(1), and $G(3,3)=0$ central components (that do not come from higher dimensional strata)
for each of these strata. Thus, $k_{LG(X)}=1$ and
$$h^{1,3}(X)=1=k_{LG(X)}.$$

\item Let $N$ be greater than $4$.
Then
$$
f_X=\frac{(x_1+x_2+1)^3}{x_1x_2y_1\cdot\ldots\cdot y_{N-2}}+y_1+\ldots+y_{N-2}
$$
and $LG_s(X)$ is given by
$$
y_0^{N-1}(x_1+x_2+x_3)^3=(\lambda y_0+y_1+\ldots+y_{N-2})y_1 \cdot \ldots \cdot y_{N-2} x_1x_2x_3.
$$
The strata we blow up are given by
$$y_1=\ldots=y_{N-2}=x_1+x_2+x_3=x_j=0,\ y_0=1.$$
In a neighborhood of such stratum $LG_s(X)$ is analytically equivalent to a hypersurface
$$a^3=\lambda y_1 \cdot\ldots\cdot y_{N-2} x_j.$$
So there are $F(3,N-2)=0$ central components of a common strata (given by
$a=y_1=\ldots=y_{N-2}=0$) of singularities for all of these hypersurfaces
by Example~\ref{example:local cubics}(1), and $G(3,N-1)=0$ central components (that do not come from higher dimensional strata)
for each of these strata. Thus, $k_{LG(X)}=0$ and
$$h^{1,N-1}(X)=0=k_{LG(X)}.$$

\end{enumerate}

\end{example}

\begin{remark}
The central fiber contains more deep information then just a Hodge number.
According to Theorem~\ref{theorem:main}, a four-dimensional cubic~$X$ has
two components of the central fiber of~$LG(X)$.
The structure of these components and their intersection
are studied in~\cite[\S4]{KP09}.
Based on this study, it was proved there that Generalized Homological
Mirror Symmetry conjecture implies irrationality of a generic cubic fourfold.
For other approaches to irrationality of a generic four-dimensional cubic, see~\cite{Kuz10}, \cite{Kul08} (cf.~\cite{AAB13}) and references therein.
\end{remark}

We finish by stating several questions related to Theorem~\ref{theorem:main}.

\begin{question}\label{question:bijection}
Our proof of Theorem~\ref{theorem:main}
relies after all on a purely combinatorial computation.
Is there a way to establish a natural bijection
between the set of relevant exceptional divisors and some
set of $(1,N-1)$-classes?
In particular, given a complete
intersection $N$-fold $X$ is there a preferred way to choose a \emph{basis}
in the space $H^{1,N-1}(X)$
corresponding to this procedure?
If this is not the case, it is still possible that a weaker version
of the question makes sense: is there a preferred family of bijections
between the set of relevant exceptional divisors
and $(1,N-1)$-classes?
In particular, to any exceptional divisor one can naturally assign a
positive integer, namely, the dimension of the center of the divisor
on $X$. Does there exist a natural filtration
on the space  $H^{1,N-1}(X)$ corresponding to this (or some other)
grading on the divisors?
\end{question}

\begin{question}
Is the filtration on the set of exceptional divisors
mentioned in Question~\ref{question:bijection}
(i.\,e., filtration by dimension of a center of an exceptional
divisor) really a natural one? Say,
is it related to the weight filtration of mixed Hodge structure
given by a sheaf of vanishing cycles (see~\cite{KKP14})?
In case it is not, an analog of Question~\ref{question:bijection}
should be asked for a more natural filtration.
\end{question}



\begin{question}
Consider a Fano $N$-fold $X$. Let $LG_s(X)$ be a (possibly singular) fiberwise
compact family admitting a crepant resolution to a $N$-dimensional Landau--Ginzburg model $LG(X)$.
Is it possible to compute $h^{1,N-1}(X)$ by some procedure taking into account only $LG_s(X)$,
not~$LG(X)$?
 It would be even more interesting
to find such procedure that is more algorithmic than computing the resolution.
(This will mostly make sense for arbitrary Fano varieties rather than complete
intersections, since in the complete intersection case Theorem~\ref{theorem:middle-Hodge-numbers} provides
an easy way to compute Hodge numbers not using a Landau--Ginzburg model at all.)
Is it possible that something
like this can be obtained using motivic integration
(see, e.\,g.,~\cite{Cr04})?
\end{question}

%

\end{document}